\documentclass[a4paper,11pt]{article}
\usepackage{amsmath}
\usepackage{amssymb}
\usepackage{amsthm}
\usepackage[all]{xy}
\usepackage[latin1]{inputenc}        
\usepackage[dvips]{graphics}
\usepackage[dvips]{graphicx}
\usepackage{amscd}
\usepackage{mathrsfs}
\usepackage[mathcal]{eucal}
\usepackage{fullpage}
\usepackage{caption}
\usepackage{float}

\author{Cinzia Bisi,\, Francesco Polizzi}
\title{Proper polynomial self-maps of  the affine space: state of the art and new results}

\date{}

\newtheorem{inizio}{Lemma}[section]
\newtheorem{theorem}[inizio]{Theorem}
\newtheorem{corollary}[inizio]{Corollary}
\newtheorem{proposition}[inizio]{Proposition}
\newtheorem{remark}[inizio]{Remark}

\newtheorem{definition}[inizio]{Definition}

\newtheorem{o-problem}[inizio]{Open Problem}

\newtheorem*{teo-L}{Theorem}
\newtheorem*{teoA}{Theorem A}
\newtheorem*{teoB}{Theorem B}
\newtheorem*{teoB1}{Theorem B1}
\newtheorem*{teoC}{Theorem C}
\newtheorem*{teoD}{Theorem D}

\newtheorem*{corollary-s}{Corollary}
\theoremstyle{definition}

\newcommand{\lr}{\longrightarrow}

\newcommand{\mC}{\mathbb{C}}

\begin{document}


\maketitle


\abstract Two proper polynomial maps $f_1, \,f_2 \colon \mC^n \lr
\mC^n$ are said to be \emph{equivalent} if there exist $\Phi_1,\,
\Phi_2 \in \textrm{Aut}(\mC^n)$ such that $f_2=\Phi_2 \circ f_1
\circ \Phi_1$. In this article we investigate proper polynomial
maps of topological degree $d \geq 2$ up to equivalence. In
particular we describe some of our recent results in the case
$n=2$ and we partially extend them in higher dimension.

\endabstract


\section{Introduction}

The semi-group of proper polynomial self-maps of the affine space
$\mathbb{A}^n$ is a basic object both in complex analysis and
algebraic geometry. It is therefore surprising how little is known
about its structure. Although there has been some progress in the
last few years, many
basic questions remain unanswered. \\
Two proper polynomial maps $f_1, \,f_2 \colon \mC^n \lr \mC^n$ are
said to be \emph{equivalent} if there exist $\Phi_1,\, \Phi_2 \in
\textrm{Aut}(\mC^n)$ such that $f_2=\Phi_2 \circ f_1 \circ
\Phi_1$. In this article we investigate proper polynomial maps of
topological degree $d \geq 2$ up to equivalence.

In Section \ref{sec:intro} we set up notation and terminology and
we state without proof some preliminary results. For further details,
we refer the reader to \cite{BP10}. \\
In Section \ref{sec:n=2} we explain our recent work in dimension
$n=2$. In \cite{Lam05} Lamy proved that any proper polynomial map
$f \colon \mC^2 \lr \mC^2$ of topological degree $2$ is equivalent
to the map $(x, \, y) \lr (x, \, y^2)$; in other words, if $d=2$
there is just one equivalence class. When $d \geq 3$ we show that
the situation is entirely different, since there are always
infinitely many equivalence classes (see Theorems A, B and B1).
Theorems A and B already appeared in our paper \cite{BP10},
whereas Theorem B1 is new. Moreover, by using Shephard-Todd's
classification of finite complex reflection groups (\cite{ST54}),
we also obtained a complete description of \emph{Galois coverings}
$f \colon \mC^2 \lr \mC^2$ up to equivalence (Theorem C).
\\ Finally, in Section \ref{sec:n=3} we give an account on the
situation in dimension $n \geq 3$ and we partially extend some of
our theorems in this setting. For instance, we prove that for $d
\geq 3$ there are still infinitely many  equivalence classes
(Theorem D). It would be certainly desirable to extend Theorem C
in higher dimension, by describing all finite Galois covers $f
\colon \mathbb{C}^n \lr \mathbb{C}^n$ up to equivalence. The main
difficulty in carrying out this project is that the linearization
theorem proven in \cite{Ka79} for $n=2$ cannot be generalized in
dimension $n \geq 3$ (see \cite{Sch89}, \cite{Kn91},
\cite{MasPet91}, \cite{MasMosPet91} for some counterexamples), so
the classification method of \cite{BP10} in this case breaks down.
Although this problem is at present far from being solved, we can
nevertheless give some partial results (see Theorem \ref{n=3-1},
Theorem \ref{n=3-2} and Remark \ref{rem}).

\bigskip \noindent
$\mathbf{Acknowledgements.}$ C. Bisi was partially supported by
Progetto MIUR di Rilevante Interesse Nazionale {\it Propriet{\`a}
geometriche
delle variet{\`a} reali e complesse} and by GNSAGA - INDAM. \\
F. Polizzi was partially supported by the World Class University
program through the National Research Foundation of Korea funded
by the Ministry of Education, Science and Technology
(R33-2008-000-10101-0). He wishes to thank the Department of
Mathematics of Sogang University (Seoul, South Korea)
and especially Yongnam Lee for the invitation and the warm hospitality. \\
Finally, F. Polizzi is indebted to J. Martin-Morales for
suggesting the use of the polynomials $Q_{d, \, \lambda}$ in the
proof of Theorem B1.

\section{Proper polynomial maps} \label{sec:intro}

\begin{definition}
Let $f \colon \mathbb{C}^n \lr \mathbb{C}^n$ be a dominant
polynomial map. We say that $f$ is \emph{proper} if it is closed and
for every point $p \in \mathbb{C}^n$ the set $f^{-1}(p)$ is compact.
Equivalently, $f$ is proper if and only if
for every compact set $K \subset \mathbb{C}^n$ the set $f^{-1}(K)$ is compact.
\end{definition}
Every proper map is necessarily surjective; the converse is not
true, for instance $(x, \, y) \lr (x+x^2y, \, y)$ provides an
example of surjective self-map of $\mathbb{C}^2$ which is not
proper. There is a purely algebraic condition for a polynomial map
to be proper, see \cite[Proposition 3]{Jel93}:
\begin{proposition} \label{prop:Jel}
A dominant polynomial map  $f \colon \mathbb{C}^n \lr \mathbb{C}^n$
is proper if and only if the push-forward map $f_{\ast} \colon
\mathbb{C}[s_1, \ldots, s_n] \lr \mathbb{C}[x_1, \ldots, x_n]$ is
finite, i.e., $f_{\ast}\mathbb{C}[s_1, \ldots, s_n] \subset
\mathbb{C}[x_1, \ldots, x_n]$ is an integral extension of rings.
\end{proposition}
We recall that if $f \colon \mathbb{C}^n \lr \mathbb{C}^n$ is the
proper polynomial map
\begin{equation*}
f(x_1, \ldots , x_n)= (f_1(x_1, \ldots ,x_n), \ldots , f_n(x_1,
\ldots , x_n)),
\end{equation*}
with $f_1, \ldots , f_n \in \mathbb{C}[x_1, \ldots , x_n],$ then
$f_{\ast}$ is defined as
\begin{equation*}
\begin{split}
f_{\ast} \colon \mathbb{C}[s_1, \ldots , s_n] & \lr \mathbb{C}[x_1, \ldots , x_n]\\
 & s_1 \lr f_1(x_1, \ldots , x_n) \\
 & \vdots \\
 & s_n \lr f_n(x_1, \ldots , x_n).
\end{split}
\end{equation*}
Moreover, if we denote by $J_f$ the determinant
of the Jacobian matrix of $f,$ then the \emph{critical locus}
$\textrm{Crit}(f)$ is defined as the affine hypersurface $V(J_f)$,
and the \emph{branch locus} $B(f)$ is the image of
$\textrm{Crit}(f)$ via $f$. The
 restriction
 \begin{equation*}
f \colon \mathbb{C}^n \setminus f^{-1}(B(f)) \lr \mathbb{C}^n
\setminus B(f)
\end{equation*}
is an unramified covering of finite degree $d$; we will call $d$
the \emph{topological degree} of $f$.

\begin{definition} \label{def:equiv}
We say that two proper polynomial maps $f_1, \,f_2 \colon \mC^n
\lr \mC^n$ are \emph{equivalent} if there exist $\Phi_1,\, \Phi_2
\in \emph{Aut}(\mC^n)$ such that
\begin{equation} \label{eq:equiv}
f_2=\Phi_2 \circ f_1  \circ \Phi_1.
\end{equation}
\end{definition}
If $f_1$ and $f_2$ are equivalent, they have the same topological
degree; moreover, the chain rule implies that $\textrm{Crit}(f_1)$
is biholomorphic to $\textrm{Crit}(f_2)$ and $B(f_1)$ is
biholomorphic to $B(f_2)$. Notice that this equivalence relation
in the semi-group of proper polynomial maps is weaker than the
conjugacy relation, in which we require $\Phi_2=\Phi_1^{-1}$. For
instance, the two maps $f_1(x, \, y)=(x, \, y^2)$ and $f_2(x, \,
y)=(x, \, y^2+x)$ are equivalent in our sense but they are not
conjugate by any automorphism of $\mathbb{C}^2$, since their sets
of fixed points are not biholomorphic. The study of conjugacy
classes of proper maps of given topological degree is certainly an
interesting problem, but we will not consider it here; some good
references are \cite{FJ07a} and \cite{FJ07b}

\section{The case $n=2$} \label{sec:n=2}
In \cite{Lam05} Lamy proved that any proper polynomial map of
topological degree $2$ is equivalent to the map $(x, \, y) \lr (x,
\, y^2)$; in other words, if $d=2$ there is just one equivalence
class. In \cite{BP10} we showed that the situation is entirely
different when $d \geq 3$; in fact, we proved the following two
results:

\begin{teoA} \label{teoA}
For every $d \geq 3$, consider the polynomial map $f_d \colon
\mathbb{C}^2 \lr \mathbb{C}^2$ given by
\begin{equation*}
f_d(x, \, y) :=(x + y + xy, \, x^{d-1}y).
\end{equation*}
Then $f$ is proper of topological degree $d$, and it is \emph{not}
equivalent to any map of the form $(x, \,y) \lr (x, \, Q(x,
\,y))$.
\end{teoA}

\begin{teoB} \label{teoB}
For all positive integers $d, \, a$, with $d \geq 3$ and $a \geq
2$, consider the polynomial map $f_{d, \,a} \colon \mathbb{C}^2
\lr \mathbb{C}^2$ given by
\begin{equation*}
f_{d, \, a}(x,\,y):=(x, \, y^d-dx^ay).
\end{equation*}
Then $f_{d, \,a}$ and $f_{d, \, b}$ are equivalent if and only if
$a=b$. It follows that if $d \geq 3$ there exist infinitely many
different equivalence classes of proper polynomial maps $f \colon
\mC^2 \lr \mC^2$ of fixed topological degree $d$.
\end{teoB}

The proof of Theorem B follows from the fact that, when $d \geq 3$
and $a \neq b$, the critical loci of $f_{d, \, a}$ and $f_{d, \,
b}$ have different Milnor number at their unique singular point $o
= (0, \, 0)$, so they cannot be biholomorphic. Theorem B provides
a \emph{discrete} family $\{f_{d, \, a} \}_{a \geq 2}$ of proper
maps of degree $d$ which are pairwise non-equivalent. Now we
refine this result, by showing the existence of a
\emph{continuous} family of maps with the same property. For all
 $d \in \mathbb{N}$, $\lambda \in \mathbb{C}$ set
 \begin{equation*}
 \begin{split}
 F_{d, \, \lambda}(x,y) & :=y^d+ \lambda x^{d-1}y + x^d, \\
 \Gamma_{d} & := \{ \lambda \in \mathbb{C} \; | \; \textrm{the polynomial } F_{d, \, \lambda}(x, \, y) \,\,\, \textrm{is square-free, i.e.
 it is the product of}  \\ & \textrm{$d$ pairwise distinct homogeneous linear factors} \}.
 \end{split}
 \end{equation*}
 One immediately sees that $\mathbb{C} \setminus \Gamma_d$ is a finite set of points, and that if $\lambda \in \Gamma_{d}$ then
 the affine variety $C_{d, \, \lambda}:=V(F_{d, \, \lambda})$ is the union of $d$ distinct lines through the origin.

 \begin{proposition} \label{Kang}
 Assume $d \geq 4$ and $\lambda, \, \mu \in \Gamma_d$.
 Then the two germs of plane curve singularities $(C_{d, \, \lambda}, \, o)$ and $(C_{d, \, \mu}, \, o)$ are analytically equivalent
if and only if $\lambda^d = \mu^d$.
\end{proposition}
\begin{proof}
See \cite[Theorems 1.3 and 2.2]{K93}
\end{proof}
Notice that the Milnor number of $C_{d, \, \lambda}$ at the origin does not depend on $\lambda$,
 since two ordinary $d$-multiple points are always topologically equivalent.
 Proposition \ref{Kang} is a particular case of a more general result saying that when $d \geq 4$
 there are infinitely many \emph{analytic} types of ordinary $d$-multiple points. For instance, if $d=4$ then
 the analytic type depends precisely on the cross-ratio
 of the four tangents, see \cite[Example 3.43.2]{GLS07}. \\
Now, setting
\begin{equation*}
  Q_{d, \, \lambda}(x, \, y):= \frac{1}{d+1}y^{d+1}+ \frac{\lambda}{2}x^{d-1}y^2 + x^dy,
\end{equation*}
we can prove
\begin{teoB1}
For all $d \geq 4$ and $\lambda \in \Gamma_d$, consider the proper polynomial map defined by
\begin{equation*}
f_{d, \, \lambda}(x, \, y):=(x, \, Q_{d, \, \lambda}(x, \, y)).
\end{equation*}
If $\lambda^d \neq \mu^d$, then $f_{d, \,  \lambda}$ and $f_{d, \, \mu}$
are \emph{not} equivalent. In particular, for all $d \geq 4$ there exist a \emph{continuous} family
of proper polynomial maps of degree $d$ whose members are pairwise non-equivalent.
\end{teoB1}
\begin{proof}
The critical locus of $f_{d, \, \lambda}$ is precisely the curve $C_{d, \, \lambda}$. Then the assertion
is an immediate
 consequence of Proposition \ref{Kang}.

\end{proof}

The previous results suggests that a satisfactory description of
all equivalence classes of proper polynomial maps $ f \colon
\mathbb{C}^2 \lr \mathbb{C}^2$ in the case $d \geq 3$ is at the
moment out of reach; nevertheless, one could hope at least to
classify those proper maps enjoying some additional property. In
\cite{BP10} we completely solved this problem in the case of
\emph{Galois coverings}; some of our computations were carried out
by using the Computer Algebra Systems \verb|GAP4| and
\verb|Singular|, see \cite{GAP4} and \cite{SING}. Let $f \colon
\mathbb{C}^2 \lr \mathbb{C}^2$ be a polynomial map which is a
Galois covering with finite Galois group $G$. Then $f$ is proper
and its topological degree equals $|G|$; moreover $G \subset
\textrm{Aut}(\mathbb{C}^2)$, and $f$ can be identified with the
quotient map $\mathbb{C}^2 \lr \mathbb{C}^2/G$. Since $G$ is a
finite group, we may assume $G \subset \textrm{GL}(2,\mathbb{C})$
by a polynomial change of coordinates (\cite[Corollary 4.4]{Ka79})
and, since $\mathbb{C}^2/G \cong \mathbb{C}^2$, it follows that
$G$ is a \emph{finite complex reflection group}. Let us denote by
$\mathbb{C}[x,\,y]^{G}$ the subalgebra of $G$-invariant
polynomials; then the following two conditions are equivalent, see
\cite[p.380]{Coh76}:
\begin{itemize}
\item[$(i)$]there are two algebraically independent homogeneous
polynomials $\phi_1, \, \phi_2 \in \mathbb{C}[x,\, y]^{G}$ which
satisfy $|G|=\textrm{deg}(\phi_1) \cdot \textrm{deg}(\phi_2)$;
\item[$(ii)$] there are two algebraically independent homogeneous
polynomials $\phi_1, \, \phi_2 \in \mathbb{C}[x, \, y]^{G}$ such
that $1$, $\phi_1$, $\phi_2$  generate
 $\mathbb{C}[x,\,y]^{G}$ as an algebra over $\mathbb{C}$.
\end{itemize}
We say that $\phi_1, \, \phi_2$ are a \emph{basic set of invariants}
for $G$. Furthermore, putting $d_1:=\textrm{deg}(\phi_1)$,
$d_2:=\textrm{deg}(\phi_2)$, the set $\{d_1, \, d_2\}$  is
independent of the particular choice of $\phi_1,\, \phi_2$. We call
$d_1$, $d_2$ the \emph{degrees} of $G$.
Complex reflection groups were
classified in all dimensions by Shephard and Todd, see
\cite{ST54} and \cite{Coh76}. Let us explain their classification
in the case $n=2$. If $G$ is reducible, i.e. if there exists a
$1$-dimensional linear subspace $V \subset \mathbb{C}^2$ which is
invariant under $G$, then we are in one of the following cases:
\begin{itemize}
\item[$(1)$] $G=\mathbb{Z}_m$, generated by $g= \left(
\begin{array}{cc}
1 & 0 \\
0 & \exp(2 \pi i/m)
\end{array}
\right);$
\item[$(2)$] $G=\mathbb{Z}_m \times \mathbb{Z}_n$, generated by \\
$g_1= \left(
\begin{array}{cc}
\exp(2 \pi i/m) & 0 \\
0 & 1
\end{array}
\right)$ \; and \; $g_2= \left(
\begin{array}{cc}
1 & 0 \\
0 & \exp(2 \pi i/n)
\end{array}
\right)$.
\end{itemize}
If $G$ is irreducible, there exists an infinite family $G(m, \, p,
\, 2)$, depending on two positive integer parameters $m$, $p$,
with $p|m$, and $19$ exceptional cases, that in \cite{ST54} are
numbered from $4$ to $22$. We start by describing the groups
belonging to the infinite family. One has
\begin{equation*}
G(m, \, p, \, 2)=\mathbb{Z}_2 \ltimes A(m, \; p, \; 2),
\end{equation*}
where $A(m, \, p, \, 2)$ is the abelian group of order $m^2/p$
whose elements are the matrices $\left(
                                  \begin{array}{cc}
                                    \theta^{\alpha_1} & 0 \\
                                    0 & \theta^{\alpha_2} \\
                                  \end{array}
                                \right)$,
with $\theta=\exp(2 \pi i/m)$ and $\alpha_1+\alpha_2 \equiv 0$
(mod $p$), whereas $\mathbb{Z}_2$ is generated by $\left(
                                  \begin{array}{cc}
                                    0 & 1 \\
                                    1 & 0 \\
                                  \end{array}
                                \right)$.
In particular, $G(m, \, m, \, 2)$ is the dihedral group of order
$2m$. \\
Now let us consider the exceptional groups in the Shephard-Todd's
list. We closely follow the treatment given in \cite{BP10}, which
was in turn inspired by \cite{ST54}. For $p=3, \, 4, \, 5$, the
abstract group
\begin{equation*}
\langle s, \, t \, | \,s^2=t^3=(st)^p=1 \rangle
\end{equation*}
is isomorphic to $\mathcal{A}_4$, $\mathcal{S}_4$ and
$\mathcal{A}_5$, respectively. These are the well-known groups of
symmetries of regular polyhedra: $\mathcal{A}_4$ is the symmetry
group of the tetrahedron, $\mathcal{S}_4$ is the symmetry group of
the cube (and of the octahedron) and $\mathcal{A}_5$ is the symmetry
group of the dodecahedron (and the icosahedron). We take Klein's
representation of these groups by complex matrices (\cite{Kl84}),
and we call $S_1$, $T_1$ the matrices corresponding to the
generators $s$ and $t$, respectively. Therefore the exceptional
finite complex reflection groups are generated by matrices
\begin{equation*}
S=\lambda S_1, \quad T=\mu T_1, \quad Z= \exp(2 \pi i /k) I,
\end{equation*}
where $\lambda$, $\mu$ are suitably chosen roots of unity and $k$ is
a suitable integer. The corresponding abstract presentations are of
the form
\begin{equation} \label{group-presentation}
\langle S, \, T, \, Z \,| \, S^2=Z^{k_1}, \, T^3=Z^{k_2}, \,
(ST)^p=Z^{k_3}, \, [S,Z]=I, \, [T, Z]=I, \, Z^k=I \rangle
\end{equation}
where $p=1, \, 2, \, 3$ and $k_1$, $k_2$, $k_3$, $k$ are suitably
chosen integers. We shall arrange the possible values of $\lambda$,
$\mu$, $k_1$, $k_2$, $k_3$, $k$ in tabular form, according to
Shephard-Todd's list (\cite[p. 280-286]{ST54}).  \\ \\
\emph{Exceptional groups derived from} $\mathcal{A}_4$. Set
$\omega=\exp(2 \pi i /3)$, $\varepsilon=\exp(2 \pi i /8)$. We have
\begin{equation*}
S_1=\left(
       \begin{array}{cc}
         i & 0 \\
         0 & -i \\
       \end{array}
     \right), \quad
T_1= \frac{1}{\sqrt{2}} \left(
       \begin{array}{cc}
         \varepsilon & \varepsilon^3 \\
         \varepsilon & \varepsilon^7 \\
       \end{array}
     \right).
\end{equation*}
The four corresponding groups are shown in Table
\ref{exceptional-from-A4} below. Here \verb|IdSmallGroup|$(G)$
denotes the label of $G$ in the \verb|GAP4| database of small
groups, which includes all groups of order less than $2000$, with the exception
of $1024$ (\cite{GAP4}). For instance, one has
\verb|[24,3]|$=\textrm{SL}_2(\mathbb{F}_3)$ and this means that
$\textrm{SL}_2(\mathbb{F}_3)$ is the third in the list of groups of
order $24$.

\begin{table}[H]
\begin{center}
\begin{tabular}{c c c c c c c c c}
\hline
$ $ & \verb|IdSmall| & $ $ & $ $ & $ $ & $ $ & $ $ & $ $ & $ $ \\
No. & \verb|Group|$(G)$ & $\lambda$ & $\mu$ & $k_1$ & $k_2$ & $k_3$
& $k$ & Degrees \\
\hline
$4$ & \verb|[24,3]| & $-1$ & $-\omega$ & $1$ & $2$ & $2$ & $2$ & $4,\, 6$ \\
$5$ & \verb|[72,25]| & $- \omega$ & $- \omega$ & $1$ & $6$ & $6$ &
$6$ &
$6, \, 12$ \\
$6$ & \verb|[48,33]| & $i$ & $- \omega$ & $4$ & $4$ & $1$ & $4$ &
$4, \, 12$ \\
$7$ & \verb|[144,157]| & $i \omega$ & $- \omega$ & $8$ & $12$ & $3$
& $12$ & $12, \, 12$ \\ \hline
\end{tabular}
\end{center}
\caption{} \label{exceptional-from-A4}
\end{table}

\noindent \emph{Exceptional groups derived from} $\mathcal{S}_4$. We
have
\begin{equation*}
S_1= \frac{1}{\sqrt{2}}\left(
       \begin{array}{cc}
         i & 1 \\
         -1 & -i \\
       \end{array}
     \right), \quad
T_1= \frac{1}{\sqrt{2}} \left(
       \begin{array}{cc}
         \varepsilon & \varepsilon \\
         \varepsilon^3 & \varepsilon^7 \\
       \end{array}
     \right).
\end{equation*}
The eight corresponding groups are shown in Table
\ref{exceptional-from-S4} below.
\begin{table}[H]
\begin{center}
\begin{tabular}{c c c c c c c c c}
\hline
$ $ & \verb|IdSmall| & $ $ & $ $ & $ $ & $ $ & $ $ & $ $ & $ $ \\
No. & \verb|Group|$(G)$ & $\lambda$ & $\mu$ & $k_1$ & $k_2$ & $k_3$
& $k$ & Degrees \\
\hline
$8$ & \verb|[96,67]| & $\varepsilon^3$ & $1$ & $1$ & $2$ & $4$ & $4$ & $8,\, 12$ \\
$9$ & \verb|[192,963]| & $i$ & $\varepsilon$ & $8$ & $7$ & $8$ & $8$ & $8,\, 24$ \\
$10$ & \verb|[288,400]| & $\varepsilon^7 \omega^2$ & $- \omega$ &
$7$ & $12$ & $12$ & $12$ & $12,\, 24$ \\
$11$ & \verb|[576,5472]| & $i$ & $ \varepsilon \omega$ &
$24$ & $21$ & $8$ & $24$ & $24,\, 24$ \\
$12$ & \verb|[48,29]| & $i$ & $1$ &
$2$ & $1$ & $1$ & $2$ & $6,\, 8$ \\
$13$ & \verb|[96,192]| & $i$ & $i$ &
$4$ & $1$ & $2$ & $4$ & $8,\, 12$ \\
$14$ & \verb|[144,122]| & $i$ & $- \omega$ &
$6$ & $6$ & $5$ & $6$ & $6,\, 24$ \\
$15$ & \verb|[288,903]| & $i$ & $i \omega$ &
$12$ & $3$ & $10$ & $12$ & $12,\, 24$ \\
 \hline
\end{tabular}
\end{center}
\caption{} \label{exceptional-from-S4}
\end{table}

\noindent \emph{Exceptional groups derived from} $\mathcal{A}_5$.
Set $\eta= \exp(2 \pi i /5)$. We have
\begin{equation*}
S_1= \frac{1}{\sqrt{5}}\left(
       \begin{array}{cc}
         \eta^4 - \eta & \eta^2 - \eta^3 \\
         \eta^2 - \eta^3 & \eta - \eta^4 \\
       \end{array}
     \right), \quad
T_1= \frac{1}{\sqrt{5}} \left(
       \begin{array}{cc}
         \eta^2 - \eta^4 & \eta^4 -1 \\
         1 - \eta & \eta^3 - \eta \\
       \end{array}
     \right).
\end{equation*}
The seven corresponding groups are shown in Table
\ref{exceptional-from-A5} below.

\begin{table}[H]
\begin{center}
\begin{tabular}{c c c c c c c c c}
\hline
$ $ & \verb|IdSmall| & $ $ & $ $ & $ $ & $ $ & $ $ & $ $ & $ $ \\
No. & \verb|Group|$(G)$ & $\lambda$ & $\mu$ & $k_1$ & $k_2$ & $k_3$
& $k$ & Degrees \\
\hline
$16$ & \verb|[600,54]| & $-\eta^3$ & $1$ & $7$ & $10$ & $10$ & $10$ & $20,\, 30$ \\
$17$ & \verb|[1200,483]| & $i$ & $i \eta^3$ & $20$ & $11$ & $20$ & $20$ & $20,\, 60$ \\
$18$ & \verb|[1800,328]| & $- \omega \eta^3$ & $\omega^2$ &
$11$ & $30$ & $30$ & $30$ & $30,\, 60$ \\
$19$ & \verb|[3600, ]| & $i \omega$ & $i \eta^3$ &
$40$ & $33$ & $40$ & $60$ & $60,\, 60$ \\
$20$ & \verb|[360,51]| & $1$ & $\omega^2$ &
$3$ & $6$ & $5$ & $6$ & $12,\, 30$ \\
$21$ & \verb|[720,420]| & $i$ & $\omega^2$ &
$12$ & $12$ & $1$ & $12$ & $12,\, 60$ \\
$22$ & \verb|[240, 93]| & $i$ & $1$ &
$4$ & $4$ & $3$ & $4$ & $12,\, 20$ \\
\hline
\end{tabular}
\end{center}
\caption{} \label{exceptional-from-A5}
\end{table}

This allows us to obtain the classification, up to equivalence, of
finite Galois coverings $f \colon \mathbb{C}^2 \lr \mathbb{C}^2$.
Set
\begin{equation*}
\begin{split}
\textsf{a}_4(x,\,y)&=x^4+(4 \xi -2)x^2y^2+y^4, \quad \xi=\exp(2 \pi i/6), \\
\textsf{b}_6(x,\, y)&=x^5y-xy^5, \\
\textsf{c}_8(x,\,y)&=x^8+14x^4y^4+y^8, \\ \textsf{d}_{12}(x, \,
y)&=
x^{12}-33x^8y^4-33x^4y^8+y^{12}, \\
\textsf{e}_{12}(x, \, y)&=x^{11}y+11x^6y^6-xy^{11}, \\
\textsf{f}_{20}(x, \,
y)&=x^{20}-228x^{15}y^5+494x^{10}y^{10}+228x^5y^{15}+y^{20}, \\
\textsf{g}_{30}(x,y)&=x^{30}+522x^{25}y^5-10005x^{20}y^{10}-10005x^{10}y^{20}-522x^5y^{25}+y^{30}.
\end{split}
\end{equation*}
Then we have

\begin{teoC} \label{teoC}
Let $f \colon \mathbb{C}^2 \lr \mathbb{C}^2$
be a polynomial map which is a Galois covering with finite Galois
group $G$. Then $f$ is equivalent to one of the normal forms
described in Table \emph{\ref{table:Galois}} below. Furthermore,
these maps are pairwise non-equivalent, with the only exception of
$\mathfrak{f}_{2, \, 1, \, 2}$ and $\mathfrak{f}_{4, \, 4, \, 2}$.
\end{teoC}
\begin{table}[H]
\begin{center}
\begin{tabular}{c c c c}
\hline
Map & $\phi_1, \, \phi_2$ & $G$ & Branch locus \\
\hline
$\mathfrak{f}_m$ & $x, \, y^m$ & $\mathbb{Z}_m$ & $y=0$ \\
$\mathfrak{f}_{m, \, n}$ & $x^m, \, y^n$ & $\mathbb{Z}_m \times
\mathbb{Z}_n$
& $xy=0$ \\
$\mathfrak{f}_{m, \, p, \, 2}$ & $x^{m/p}y^{m/p},\, x^m + y^m$ &
$G(m, \, p,\,
2)$ &  $x(y^2-4x^p)=0 \quad \textrm{if} \; p \neq m$ \\
 & & & $\quad \quad y^2-4x^p=0 \quad \textrm{if} \; p=m$ \\
$\tilde{\mathfrak{f}}_4$ & $\textsf{a}_4, \, \textsf{b}_6$ & $G_4=$\verb|[24, 3]| & $x^3+(-24\xi+12)y^2=0$ \\
$\tilde{\mathfrak{f}}_5$ & $\textsf{b}_6, \, (\textsf{a}_4)^3$ & $G_5=$\verb|[72, 25]| & $y(x^2+ (\frac{1}{18\xi}-\frac{1}{36})y)=0$  \\
$\tilde{\mathfrak{f}}_6$ & $\textsf{a}_4, \, (\textsf{b}_6)^2$ & $G_6=$\verb|[48, 33]| &  $y(x^3+(-24\xi+12)y^2)=0$\\
$\tilde{\mathfrak{f}}_7$ & $(\textsf{b}_6)^2, \, (\textsf{a}_4)^3$ & $G_7=$ \verb|[144, 157]| & $xy(x+ (\frac{1}{18\xi}-\frac{1}{36})y)=0$\\
$\tilde{\mathfrak{f}}_8$ & $\textsf{c}_8, \, \textsf{d}_{12}$ & $G_8=$\verb|[96, 67]| & $y^2-x^3=0$\\
$\tilde{\mathfrak{f}}_9$ & $\textsf{c}_8, \, (\textsf{d}_{12})^2$ & $G_9=$ \verb|[192, 963]| & $y(y-x^3)=0$ \\
$\tilde{\mathfrak{f}}_{10}$ & $\textsf{d}_{12}, \, (\textsf{c}_8)^3$ & $G_{10}=$\verb|[288, 400]| & $y(y-x^2)$=0 \\
$\tilde{\mathfrak{f}}_{11}$ & $(\textsf{d}_{12})^2, \, (\textsf{c}_8)^3$ & $G_{11}=$\verb|[576, 5472]| & $xy(x-y)=0$\\
$\tilde{\mathfrak{f}}_{12}$ & $\textsf{b}_6, \, \textsf{c}_8$ & $G_{12}=$\verb|[48, 29]| & $y^3-108x^4=0$\\
$\tilde{\mathfrak{f}}_{13}$ & $\textsf{c}_8, \, (\textsf{b}_6)^2$ & $G_{13}=$\verb|[96, 192]| & $y(x^3-108y^2)$=0 \\
$\tilde{\mathfrak{f}}_{14}$ & $\textsf{b}_6, \, (\textsf{d}_{12})^2$ & $G_{14}=$\verb|[144, 122]| & $y(y+108x^4)$=0 \\
$\tilde{\mathfrak{f}}_{15}$ & $(\textsf{b}_6)^2, \, (\textsf{d}_{12})^2$ & $G_{15}=$ \verb|[288, 903]| & $xy(y+108x^2)=0$  \\
$\tilde{\mathfrak{f}}_{16}$ & $\textsf{f}_{20}, \, \textsf{g}_{30}$ & $G_{16}=$\verb|[600, 54]| &  $y^2-x^3=0$\\
$\tilde{\mathfrak{f}}_{17}$ & $\textsf{f}_{20}, \, (\textsf{g}_{30})^2$ & $G_{17}=$\verb|[1200, 483]| & $y(y-x^3)=0$\\
$\tilde{\mathfrak{f}}_{18}$ & $\textsf{g}_{30}, \, (\textsf{f}_{20})^3$ & $G_{18}=$\verb|[1800, 328]| & $y(y-x^2)=0$\\
$\tilde{\mathfrak{f}}_{19}$ & $(\textsf{g}_{30})^2, \, (\textsf{f}_{20})^3$ & $G_{19}=$ \verb|[3600, ]| & $xy(x-y)=0$ \\
$\tilde{\mathfrak{f}}_{20}$ & $\textsf{e}_{12}, \, \textsf{g}_{30}$ & $G_{20}=$ \verb|[360, 51]| & $y^2-1728 x^5=0$\\
$\tilde{\mathfrak{f}}_{21}$ & $\textsf{e}_{12}, \, (\textsf{g}_{30})^2$ & $G_{21}=$\verb|[720, 420]| & $y(y-1728x^5)=0$\\
$\tilde{\mathfrak{f}}_{22}$ & $\textsf{e}_{12}, \, \textsf{f}_{20}$ & $G_{22}=$\verb|[240, 93]| & $y^3+1728x^5=0$ \\

\hline

\end{tabular}
\end{center}
\caption{} \label{table:Galois}
\end{table}

The following corollary is a generalization of Lamy's result to
the case of Galois coverings of arbitrary degree.

\begin{corollary} \label{cor:finite-galois}
For all $d \geq 2$, there exist only finitely many equivalence
classes of Galois coverings $f \colon \mathbb{C}^2 \lr
\mathbb{C}^2$
 of topological degree $d$.
\end{corollary}

\section{The case $n \geq 3$} \label{sec:n=3}

We have only few general results about proper polynomial self-maps
of $\mathbb{C}^n$ for $n \geq 3$. First of all, we can prove the
following analogue of Theorem B:
\begin{teoD}
Let $\emph{\textbf{a}}:=(a_1, \ldots, a_{n-1}) \in \mathbb{N}^{n-1}$ be
such that $a_i \geq 2$ for all $i$. For all $d \geq 3,$ consider
the proper polynomial map  $f_{d, \, \emph{\textbf{a}}} \colon
\mathbb{C}^n \lr \mC^n$ defined by
\begin{equation*}
f_{d, \, \emph{\textbf{a}}}(x_1, \ldots, x_n):= (x_1,\, x_2, \ldots,
x_{n-1}, \, x_n^d-d(x_1^{a_1}+x_2^{a_2}+ \cdots
+x_{n-1}^{a_{n-1}})x_n).
\end{equation*}
If $\prod_{i=1}^{n-1}(a_i-1) \neq \prod_{i=1}^{n-1}(b_i-1)$ then
$f_{d, \, \emph{\textbf{a}}}$ and $f_{d, \, \emph{\textbf{b}}}$ are \emph{not}
equivalent. It follows that for all $d \geq 3$ there exist
infinitely many different equivalence classes of proper polynomial
maps $f \colon \mC^n \lr \mC^n$ of topological degree $d$.
\end{teoD}
\begin{proof}
The critical locus of $f_{d, \, \textbf{a}}$ is the affine
hypersurface $C_{d, \textbf{a}}$ of equation
$x_n^{d-1}-x_1^{a_1}-x_2^{a_2}- x_{n-1}^{a_{n-1}}=0$, whose unique
singular point is $o:=(0, \ldots, 0)$. The Milnor number of $C_{d,
\, \textbf{a}}$ in $o$ is
\begin{equation*}
\mu(C_{d, \, \textbf{a}}, \, o)= (d-2)\prod_{i=1}^{n-1}(a_i-1).
\end{equation*}
It follows that if $d \geq 3$ and $\prod_{i=1}^{n-1}(a_i-1) \neq
\prod_{i=1}^{n-1}(b_i-1)$ then $C_{d, \, \textbf{a}}$ and $C_{d,
\, \textbf{b}}$ are not biholomorphic, hence $f_{d, \,
\textbf{a}}$ and $f_{d, \, \textbf{b}}$ are not equivalent.
\end{proof}
It would be also desirable to extend Theorem C in higher
dimension, in other words to classify all the finite Galois covers
$f \colon \mathbb{C}^n \lr \mathbb{C}^n$ up to equivalence. The
main difficulty in carrying out this project is that the
linearization theorem stated in \cite{Ka79} for $n=2$ cannot be
generalized in dimension $n \geq 3$. So the classification method
of \cite{BP10} in this case breaks down. For the reader's
convenience, let us give a short account on these topics; for
further details we refer to the survey paper \cite{Kr95}. \\ In
\cite{Ka79} it was conjectured that if $G$ is a linearly reductive
algebraic group acting regularly on $\mathbb{C}^n$, then $G$ has a
fixed point, say $p$, and the action of $G$ is linear with respect
to a suitable coordinate system of $\mathbb{C}^n$ having $p$ as
its origin (the so-called Algebraic Linearization Conjecture). The
first results in this direction were very promising, indeed any
such action on $\mathbb{C}^2$ is linearizable as a consequence of
the Jung's Theorem on the structure $\textrm{Aut}(\mC^2)$. Any
torus action with an orbit of codimension one is linearizable by
Bialynicki-Birula, see \cite{BiBi66}, \cite{BiBi67}, and Kraft,
Popov and Panyushev showed that every semisimple group action is
linearizable on $\mathbb{C}^3$ and $\mathbb{C}^4$, see \cite{KrP85} and \cite{Pa84}.\\
On the other hand, in 1989 Schwarz discovered the first examples
of non-linearizable actions of the orthogonal group $O(2)$ on
$\mathbb{C}^4$ and of $SL_2$ on $\mathbb{C}^7,$ \cite{Sch89}.
Using these results, Knop showed that every connected reductive
group which is not a torus admits a faithful non-linearizable
action on some affine space $\mathbb{C}^n,$ \cite{Kn91}. Using a
different approach, Masuda, Moser-Jauslin and Petrie produced more
examples and discovered the first non-linearizable actions of
finite groups, namely dihedral groups of order $\ge 10$ on
$\mathbb{C}^4,$ see \cite{MasMosPet91}. So far, all these examples
of non-linearizable actions have been obtained from non-trivial
$G$-vector bundles on representation spaces $V$ of $G$ using an
idea of Bass and Haboush: for example in \cite{MasMosPet91} it is
proven that if $G$ is a dihedral group of order $\ge 10,$ then
there exists a positive-dimensional continuous family of
isomorphism classes of $G$-vector bundles to which corresponds a
positive-dimensional continuous family of inequivalent actions on
$\mathbb{C}^4.$ This method does not work in the holomorphic
setting, however in \cite{DerKut98} it is shown how to construct
non-linearizable holomorphic actions on $\mathbb{C}^n$
for all reductive groups. \\
 These results are not conclusive,
and in particular the problem of describing all finite,
non-linearizable automorphism subgroups of $\mathbb{A}^n$ for $n \geq
3$ is at
 present far from being solved. For instance, it is not even known
 whether there exist non-linearizable \emph{involutions} on
 $\mathbb{A}^3$. \\
It is not our purpose to investigate these deep questions here, so
we just present the following two results:

\begin{theorem} \label{n=3-1}
Let $n\ge2$ and $f: \mathbb{C}^n \lr \mathbb{C}^n$ be a polynomial
map which is a Galois covering with finite Galois group $G \cong
\mathbb{Z}_m = \langle \sigma \rangle,$ where $\sigma$ is a
triangular automorphism of $\mathbb{C}^n$ of the form
\begin{equation*}
\sigma (x_1, \cdots, x_n)=(s_1 x_1 + a_1, s_2 x_2 + a_2 (x_1),
\cdots, s_n x_n + a_n (x_1, \cdots , x_{n-1})), \; \; s_i \in
\mathbb{C}^*
\end{equation*}
such that $\sigma^m=I.$ Then $f$ is equivalent to $\mathfrak{f}_m (x_1,
\cdots, x_n)=(x_1, \, x_2, \cdots, \, x_{n-1}, \, x_n^m).$
\end{theorem}
\begin{proof}
By \cite{Ivan98} the group generator $\sigma$ is linearizable,
 so the group action is also
linearizable. By using Shephard-Todd's classification of finite
complex reflection groups, we see that $G$ is conjugated in $U(n)$
to the group generated by $\tilde{\sigma}(x_1, \cdots, x_n) =(
x_1, \, x_2, \cdots, x_{n-1}, \, \theta_m x_n),$ where $\theta_m$
is a primitive $m$-th root of unity.
\end{proof}

\begin{theorem} \label{n=3-2}
Let $f: \mathbb{C}^3 \lr \mathbb{C}^3$ be a polynomial map which
is a Galois covering with finite Galois group $G$, and assume that
the action of $G$ is linearizable and reducible. Then $G$ is one
of the groups in Table \emph{\ref{table:Galois}} and we are in one
of the following cases$:$
\begin{itemize}
\item[$(1)$] $f$ is equivalent to the map $(x_1, \,
\mathfrak{f}(x_2,x_3))$, where $\mathfrak{f}$ is the normal form
on $\mathbb{C}^2$ corresponding to $G;$
 \item[$(2)$]  $f$ is
equivalent to the map $(\theta_m x_1, \, \mathfrak{f}(x_2,x_3))$,
where
$\mathfrak{f}$ is the normal form on $\mathbb{C}^2$ corresponding to $G$ and $\theta_m$ is a primitive $m-$th root of unity. \\
\end{itemize}
\end{theorem}
\begin{proof}
Since the action is reducible, there exists either a $1-$dimensional or a $2-$dimensional
linear subspace $V \subset \mathbb{C}^3$ which is invariant under
$G;$ then its orthogonal complement $V^{\perp}$ is also invariant,
see \cite{Se71}, and up to a linear change of coordinates we may
assume $V=\langle e_1 \rangle,$ $V^{\perp}=\langle e_2, e_3
\rangle$ where $\{ e_1, e_2, e_3 \}$ is the canonical basis of
$\mathbb{C}^3.$ Then the assertion follows by using the
classification given in Theorem C.
\end{proof}

\begin{remark} \label{rem}
By using the same methods of \emph{\cite{BP10}}, it is possible to
completely classify the Galois coverings $f \colon \mathbb{C}^n
\lr \mathbb{C}^n$ such that the $G$-action on $\mathbb{C}^n$ is
\emph{linearizable}. Indeed, this is equivalent to compute a
minimal base of generators of the invariant algebra
$\mathbb{C}[x_1, \ldots, x_n]^G$ for each of the $34$ exceptional
groups in the Shephard-Todd's list. This is a standard calculation
that can be carried out by using either invariant theory $($as in
\emph{\cite{ST54}}$)$ or some Computer Algebra Systems $($e.g.
\verb|GAP4| and \verb|Singular|$)$. However, some of these groups
have very large order $($for instance, in the last case of the
list we have $G=W(E_8)$, whose order is $696729600)$, so the
problem is computationally hard and we think that the outcome is
not worthy of the effort.
\end{remark}



\bigskip \bigskip
CINZIA BISI \\
Dipartimento di Matematica, Universit\`a di Ferrara, Via Machiavelli n. 35, \\
44121 Ferrara (FE), Italy. \\
\emph{E-mail address}: \verb|bsicnz@unife.it| \\ \\

FRANCESCO POLIZZI \\
Dipartimento di Matematica, Universit\`a della Calabria, Via P. Bucci
Cubo 30B, \\
87036 Arcavacata di Rende (CS), Italy. \\
\emph{E-mail address}: \verb|polizzi@mat.unical.it|

\end{document}